\documentclass[a4paper,reqno,12pt]{amsart} 
\usepackage{amsmath} 

\usepackage{amssymb}      

\usepackage{yhmath}
\usepackage{mathdots}
\usepackage{MnSymbol}

\usepackage{amsthm}      

\usepackage{tikz,amsmath}
\usetikzlibrary{arrows}

\usepackage{epsfig}      

\usepackage{graphicx}
\usepackage[normalem]{ulem}
\usepackage[all,line,
]{xy} 
\CompileMatrices 

\newcommand{\Span}{\operatorname{Span}}

 \newcommand{\supp}{\operatorname{supp}}

   \theoremstyle{plain}
   \newtheorem{thm}{Theorem}[section]
   \newtheorem{prop}[thm]{Proposition}
   \newtheorem{lemma}[thm]{Lemma}  
   \newtheorem{cor}[thm]{Corollary}
   \theoremstyle{definition}
   
   \newtheorem{defn}[thm]{Definition}
   
   \theoremstyle{remark}

\usepackage{mdframed,xcolor}

\definecolor{mybgcolor}{gray}{0.8}
\definecolor{myframecolor}{rgb}{.647,.129,.149}

\mdfdefinestyle{mystyle}{
  usetwoside=false,
  skipabove=0.6em plus 0.8em minus 0.2em,
  skipbelow=0.6em plus 0.8em minus 0.2em,
  innerleftmargin=.25em,
  innerrightmargin=0.25em,
  innertopmargin=0.25em,
  innerbottommargin=0.25em,
  leftmargin=-.75em,
  rightmargin=-0em,
  topline=false,
  rightline=false,
  bottomline=false,
  leftline=false,
  backgroundcolor=mybgcolor,
  splittopskip=0.75em,
  splitbottomskip=0.25em,
  innerleftmargin=0.5em,
  leftline=true,
  linecolor=myframecolor,
  linewidth=0.25em,
}

\newmdenv[style=mystyle]{important}

\usepackage{lipsum}

   \numberwithin{equation}{section}

        \date{\today}

\title[Random walks on groups and KMS states]{Random walks on groups and KMS states}

\author{Johannes Christensen and Klaus Thomsen}

\date{\today}

\email{ johannes@math.au.dk, matkt@math.au.dk}
\address{Department of Mathematics, Aarhus University, Ny Munkegade, 8000 Aarhus C, Denmark}

\begin{document}

\maketitle

\begin{abstract} A classical construction associates to a transient random walk on a discrete group $\Gamma$ a compact $\Gamma$-space $\partial_M \Gamma$ known as the Martin boundary. The resulting crossed product $C^*$-algebra $C(\partial_M \Gamma) \rtimes_r \Gamma$ comes equipped with a one-parameter group of automorphisms given by the Martin kernels that define the Martin boundary. In this paper we study the KMS states for this flow and obtain a complete description when the Poisson boundary of the random walk is trivial and when $\Gamma$ is a torsion free non-elementary hyperbolic group. We also construct examples to show that the structure of the KMS states can be more complicated beyond these cases.
 \end{abstract}

\section{Introduction} The purpose of this paper is to show that random walks on groups in a natural way gives rise to $C^*$-algebras equipped with a one-parameter flow of automorphisms and hence to possible models in quantum statistical mechanics, \cite{BR}, and to initiate the study of the equilibrium states in such models. Our approach is based on results and examples from the research on random walks.

The data underlying a random walk on a discrete countable group $\Gamma$ is a map 
$$
\mu : \Gamma \to [0,1]
$$
 such that $\sum_{g \in \Gamma} \mu(g) = 1$ and the support of $\mu$ generates $\Gamma$ as a semigroup. In the random walk governed by $(\Gamma,\mu)$ the value $\mu(g^{-1}h)$ is the probability of going from $g \in \Gamma$ to $h \in \Gamma$. For transient random walks, which are those that almost certainly diverge to infinity, there is a compactification of the group whose boundary $\partial_M\Gamma$ is called the Martin boundary and which is used for the study of the $\mu$-harmonic functions on $\Gamma$. There is an action of $\Gamma$ on the Martin boundary which extends the left action of $\Gamma$ on itself and we can therefore introduce the full and reduced crossed product $C^*$-algebras $C(\partial_M \Gamma) \rtimes \Gamma$ and $C(\partial_M \Gamma) \rtimes_r \Gamma$. The Martin boundary is constructed from the Martin kernel $K : \Gamma \times \Gamma \to \mathbb R$ which is defined such that
$$
K(g,h) = \frac{F(g,h)}{F(e,h)}
$$
where $e$ is the neutral element and $F(g,h)$ is the probability of visiting $h$ when starting at $g$. The Martin compactification is the smallest compactification of $\Gamma$ with the property that the function $h \mapsto K(g,h)$ extends to a continuous function on the compactification for every $g\in \Gamma$. The Martin boundary $\partial_M \Gamma$ is the complement of $\Gamma$ in this compactification and the Martin kernel extends by continuity to a function
$$
K : \Gamma \times \partial_M \Gamma \ \to \ \mathbb R \ .
$$
By definition the probability of going in one step from $x$ to $y$ in $\Gamma$ is the same as that of going from $gx$ to $gy$ for each $g \in \Gamma$ and it follows therefore that $F(gx,gy) = F(x,y)$. A straightforward consequence of this and the definiton of $K$ is  the cocycle relation 
\begin{equation*}\label{04-06-19}
K(g^{-1},hy)K(h^{-1},y) \ = \ K(h^{-1}g^{-1},y) \ ,
\end{equation*}
valid for all $g,h,y \in \Gamma$, and by continuity also when $y\in \partial_M \Gamma$. This cocycle relation is exactly what is needed to ensure that we can define a flow $\alpha^{\mu}_t, \ t \in \mathbb R$, of automorphisms both on $C(\partial_M\Gamma) \rtimes \Gamma$ and on $C(\partial_M\Gamma) \rtimes_r \Gamma$ such that
$$
\alpha^{\mu}_t(fU_g) \ = \ fU_gK({g}^{-1}, \ \cdot)^{-it}  \ ,
$$
when $f \in C(\partial_M \Gamma)$ and $U$ is the canonical unitary representation of $\Gamma$ in the crossed product(s). 

In this way the Martin boundary not only comes equipped with an action of $\Gamma$ so that we can consider the associated crossed product $C^*$-algebras $C(\partial_M\Gamma) \rtimes \Gamma$ and $C(\partial_M\Gamma) \rtimes_r \Gamma$ but also with a cocycle arising naturally from the ingredients of its construction which defines flows on these algebras. We are here only interested in the reduced crossed product, but we carry the full crossed product along in order to be able to use Neshveyev's results, \cite{N}.

Instead of plunging directly to random walks we introduce the flows we shall consider in a more general setting where it becomes apparent from the work of Rieffel \cite{Ri} that they can be considered as examples of the geodesic flows which were introduced by Connes in \cite{C}. It follows that in terms of noncommutative geometry, what we establish here is a relation between the stationary measures of a random walk and the equilibrium states of the geodesic flow on the cosphere algebra of a spectral triple which can be defined from the random walk by the methods developed by Rieffel.

Let us describe our results in the case where the random walk has finite support. It turns out that the set of possible inverse temperatures, by which we here mean the set of non-zero real numbers $\beta$ for which there exists a $\beta$-KMS state for the flow $\alpha^{\mu}$, only depends on whether or not the Martin boundary contains what we call a spine; an element $\xi_0 \in \partial_M \Gamma$ such that $K(g,\xi_0) = 1$ for all $g \in \Gamma$. If it does,  the set of non-zero inverse temperatures is $\mathbb R \backslash \{0\}$, and if it does not the set consists only of the number $1$. If there is a spine and $\beta \notin \{0,1\}$, the simplex of $\beta$-KMS states is affinely homeomorphic to the tracial state space of $C^*_r(\Gamma)$. So what remains to be figured out is what the simplex of $1$-KMS states is, both when there is a spine and when there is not. It is always non-empty and we show that it only contains one element when the Poisson boundary of the random walk is trivial and when $\Gamma$ is a torsion free non-elementary hyperbolic group. But we also show by examples that it can contain arbitrarily many extreme points. It appears that the structure of the $1$-KMS states of $\alpha^{\mu}$ is very case-sensitive and at present it seems hopelessly difficult to find a description of it in the general case.

The paper is aimed primarily at operator algebraists, but we hope that experts in random walks on groups will find the examples constructed in the last section of interest. Also it would be nice if the paper could trigger research into the question about existence of a spine in the Martin boundary because our results show that the structure of KMS states and the possible inverse temperatures is governed by this. An intriguing observation in this respect is that there does not seem to be any random walk on an amenable group without a spine in the Martin boundary and also no example, as far as we can tell, of a random walk on a non-amenable group for which there is a spine in the Martin boundary.



\bigskip

\emph{Acknowledgement} The work was supported by the DFF-Research Project 2 `Automorphisms and Invariants of Operator Algebras', no. 7014-00145B.

\section{Generalized gauge actions on crossed products} Let $X$ be a compact metric space, $\Gamma$ a countable discrete group and $\Gamma \ni g \mapsto \phi_g$ a representation of $\Gamma$ by homeomorphisms of $X$; i.e. $\phi_g : X \to X$ is a homeomorphism and $\phi_{gh} = \phi_g \circ \phi_h$ for all $g,h \in \Gamma$. There is then an action $\gamma = \left(\gamma_g\right)_{g \in \Gamma}$ of $\Gamma$ by automorphisms of $C(X)$ defined such that
$$
\gamma_g(f) = f \circ \phi_g^{-1} = f \circ \phi_{g^{-1}} \ 
$$
for all $f \in C(X)$. We shall be concerned with the crossed product $C^*$-algebras of the pair $(X,\phi)$ and mainly the reduced crossed product $C(X)\rtimes_r \Gamma$, cf. \cite{Pe}. But the full crossed product $C(X)\rtimes \Gamma$ will also play a role. A map $D : \Gamma \to C(X)$ is a cocycle or a $\phi$-cocycle on $X$ when $D_g \in C(X)$ is real-valued and
\begin{equation*}\label{cocycle}
 D_g \circ \phi_h \ + \ D_h \ = \ D_{gh}
\end{equation*}
for all $g,h \in  \Gamma$. Let $\{U_g\}_{g \in \Gamma}$ denote the canonical unitary representation of $\Gamma$, both in $C(X)\rtimes_{r} \Gamma$ and in $C(X)\rtimes \Gamma$. In particular, $U_{g}fU_{g}^{*}=\gamma_{g}(f)$ for all $f\in C(X)$ and $g\in \Gamma$. The defining properties of a cocycle ensures that there are continuous flows $\alpha^D$ on $C(X)\rtimes_{r} \Gamma$ and $C(X)\rtimes \Gamma$ defined such that
$$
\alpha^D_t(fU_g) = fU_ge^{-it D_g} \ 
$$
for all $t\in \mathbb R$, all $g\in \Gamma$ and all $f \in C(X)$. The canonical surjection $C(X)\rtimes \Gamma \to C(X)\rtimes_r \Gamma$ is then equivariant. In the following we will simultaneously consider $(C(X)\rtimes_r\Gamma, \alpha^D)$ and $(C(X)\rtimes \Gamma, \alpha^D)$ and only specify which when it is necessary.

\subsection{KMS states and conformal measures}
Let $\beta \in \mathbb R$. A state $\omega$ on $C(X)\rtimes_{r} \Gamma$, resp. $C(X)\rtimes \Gamma$, is a $\beta$-KMS state for $\alpha^D$ when 
$$
\omega(ab) \ = \ \omega(b\alpha^D_{i\beta}(a))
$$
for all $\alpha^D$-analytic elements $a,b$ in $C(X)\rtimes_{r} \Gamma$, resp. $C(X)\rtimes \Gamma$, \cite{BR}. The values of $\beta$ for which there is a $\beta$-KMS state will be called the KMS spectrum of $\alpha^D$. In both algebras the subspace
$$
\Span \left\{ f U_g: \ g \in \Gamma, \ f \in C(X) \right\}
$$
is a dense $\alpha^{D}$-invariant $*$-subalgebra consisting of $\alpha^D$-analytic elements, and hence a state $\omega$ is a $\beta$-KMS state for $\alpha^D$ iff
$$
\omega\left( f_1U_{g_1}f_2U_{g_2}\right) \ = \ \omega(f_2U_{g_2}\alpha^D_{i\beta}(f_1U_{g_1}) \ = \  \omega(f_2U_{g_2}f_1U_{g_1}e^{\beta D_{g_1}})
$$
for all $f_1,f_2 \in C(X), \ g_1,g_2 \in \Gamma$. Let $m$ be a Borel probability measure on $X$. For $\beta \in \mathbb R$ we say that $m$ is $e^{\beta D}$-conformal for $\phi$ when
\begin{equation}\label{14-03-19a}
m\left(\phi_g(B)\right) = \int_B e^{\beta D_g} \ \mathrm{d} m
\end{equation}
for all Borel subsets $B \subseteq X$ and all $g \in \Gamma$. In terms of Radon-Nikodym derivatives the condition is that the measure $m \circ \phi_g$, defined such that $m \circ \phi_g(B) = m\left(\phi_g(B)\right)$, is absolutely continuous with respect to $m$ with Radon-Nikodym derivative
\begin{equation*}\label{21-03-19a}
\frac{\mathrm{d} m \circ \phi_g}{\mathrm{d} m} \ = \ e^{\beta D_g}
\end{equation*}
for all $g \in \Gamma$. Alternatively, as is easily verified, $m$ is $e^{\beta D}$-conformal if and only if
\begin{equation}\label{20-12-18fx}
\int_X f \ \mathrm{d} m = \int_X f\circ \phi_g \ e^{\beta D_g} \ \mathrm{d} m 
\end{equation}
for all $f \in C(X)$ and all $g \in \Gamma$. In particular, an $e^{\beta D}$-conformal measure $m$ satisfies that
\begin{equation}\label{21-01-20b}
\int_X e^{\beta D_g} \ \mathrm{d}m \ = \ 1 \ 
\end{equation} 
for all $g \in \Gamma$.

When $\omega$ is {\color{red}a} $\beta$-KMS state for $\alpha^D$ we find that
\begin{equation*}
\begin{split}
&\omega(f) = \omega(U_g(f\circ \phi_g) U_g^*) = \omega((f\circ \phi_g) U_g^*\alpha^D_{i\beta}(U_g))\\
& = \omega((f\circ \phi_g) U_g^*U_ge^{\beta D_g}) = \omega((f\circ \phi_g) e^{\beta D_g}) 
\end{split}
\end{equation*}
for all $f \in C(X)$ and $g \in \Gamma$. By comparing with \eqref{20-12-18fx} we see that the restriction of $\omega$ to $C(X)$ gives rise, via the Riesz representation theorem, to an $e^{\beta D}$-conformal measure on $X$. In particular, there are no $\beta$-KMS states for $\alpha^D$ unless there is an $e^{\beta D}$-conformal measure for $\phi$.

To go in the other direction, from measures on $X$ to states on $C(X)\rtimes_{r} \Gamma$, let $Q : C(X)\rtimes_{r} \Gamma \to C(X)$ be the canonical conditional expectation. Composed with the map $C(X)\rtimes \Gamma \to C(X) \rtimes_r \Gamma$ we can also consider $Q$ as a conditional expectation on $C(X)\rtimes \Gamma$. Then any Borel probability measure $m$ on $X$ gives rise to a state $\omega_m$ on both $C(X)\rtimes_{r} \Gamma$ and $C(X)\rtimes \Gamma$ defined such that
\begin{equation}\label{23-01-19cx}
\omega_m (a) \ = \ \int_X Q(a) \ \mathrm{d} m \ .
\end{equation}

\begin{lemma}\label{23-01-19ax} A Borel probability measure $m$ on $X$ is an $e^{\beta D}$-conformal measure iff the state $\omega_m$ defined by \eqref{23-01-19cx} is a $\beta$-KMS state.
\end{lemma}
\begin{proof} The 'if' part follows from the preceding, so assume that $m$ is $e^{\beta D}$-conformal. It suffices to show that
$$
\omega_m(fU_ghU_k) = \omega_m\left(hU_kfU_ge^{\beta D_g}\right)
$$
when $f,h \in C(X)$ and $g,k \in \Gamma$. Since
$$
Q \left(f U_g hU_k\right) \ = \ Q\left(hU_kfU_ge^{\beta D_g}\right)  \ = \ 0
$$
when $k \neq g^{-1}$, the equality holds trivially in this case. It suffices therefore to consider the case $k = g^{-1}$ where we get
\begin{equation*}
\begin{split}
&\omega_m\left( fU_g h U_{g^{-1}}\right) = \omega_m\left( f h \circ \phi_g^{-1}\right)   =\omega_m\left(  (f \circ \phi_g) h  e^{\beta D_g}\right) \\ &  = \omega_m\left(h f \circ  \phi_g e^{\beta D_g}\right)  =  \omega_m\left( h U_{g^{-1}}fU_g e^{\beta D_g}\right) \ ,
\end{split}
\end{equation*}
which is the required equation.
\end{proof}

Thus the map from $\beta$-KMS states for $\alpha^D$ to $e^{\beta D}$-conformal measures for $\phi$, given by restriction of states to $C(X)$, is an affine surjection. It is not hard to see that the map is injective and hence an affine homeomorphism when $\phi$ is a free action, but in general it is not. 

We note that ergodicity of an $e^{\beta D}$-conformal measure is equivalent to the extremality in the compact convex set of $e^{\beta D}$-conformal measures as well as to an extremality condition on the corresponding $\beta$-KMS state. We leave the proof to the reader.

\begin{thm}\label{11-03-20} Let $m$ be a $e^{\beta D}$-conformal measure. The following are equivalent:
\begin{itemize}
\item $m$ is ergodic for $\phi$.
\item $m$ is extremal in the compact convex set of $e^{\beta D}$-conformal measures.
\item The set of $\beta$-KMS-states $\nu$ for $\alpha^D$ for which
$$
\nu(f) = \int_X f \ \mathrm{d} m \ \ \forall f \in C(X) \ 
$$ is a closed face in the simplex of $\beta$-KMS states for $\alpha^D$.
\end{itemize}
\end{thm}

\section{Flows from group compactifications}\label{bussec} 

Let $\Gamma$ be a countable discrete group. Following \cite{Ri} we say that a function $\omega: \Gamma  \to \mathbb R$ is translation bounded when
\begin{equation}\label{25-03-19a}
\sup_{h \in \Gamma}  \ \left|\omega(gh) \ - \ \omega(h)\right| \ < \ \infty \  \ \ \forall g \in \Gamma  \ .
\end{equation}
Starting from a translation bounded function Rieffel gave in \cite{Ri} a construction which we now recall. Let $\mathcal B(\Gamma)$ denote the $C^*$-algebra of bounded functions on $\Gamma$ and let $1_g \in \mathcal B(\Gamma)$ be the characteristic function of the set $\{g\}$. We let $\mathcal B_{\omega}$ be the $C^*$-subalgebra of $\mathcal B(\Gamma)$ generated by the constant functions, the functions $1_g,  g \in \Gamma$, and the functions $\Delta_g$, $g \in \Gamma$, where 
\begin{equation}\label{28-05-19d}
\Delta_g(h) = \omega(gh) - \omega(h) \ . 
\end{equation}
$\mathcal B_{\omega}$ is a separable unital and abelian $C^*$-algebra and we denote its space of characters by $\overline{\Gamma}^{\omega}$. By Gelfand's theorem we can identify $\mathcal B_{\omega}$ with $C\left(\overline{\Gamma}^{\omega}\right)$. Every group element $g \in \Gamma$ defines a character $c_g$ on $\mathcal B_{\omega}$ such that $c_g(f) = f(g)$, giving rise to an embedding $\Gamma \subseteq \overline{\Gamma}^{\omega}$ such that $\Gamma$ is dense, open and discrete in $\overline{\Gamma}^{\omega}$. Thus $\overline{\Gamma}^{\omega}$ is a compactification $\Gamma$, called the $\omega$-compactification by Rieffel in \cite{Ri}.
The complement 
$$
\partial_{\omega} \Gamma \ = \ \overline{\Gamma}^{\omega} \backslash \Gamma
$$
is a compact metric space and following Rieffel we call it the $\omega$-boundary of $\Gamma$. As a subset of $\overline{\Gamma}^{\omega}$ the elements of $\partial_{\omega} \Gamma$ are the characters $\kappa$ on $\mathcal B_{\omega}$ that annihilate all $1_g$, i.e.
\begin{equation*}\label{27-03-19}
\partial_{\omega} \Gamma \ = \ \left\{ \kappa \in \overline{\Gamma}^{\omega}: \ \kappa(1_g) = 0 \ \forall g \in \Gamma \right\} \ .
\end{equation*}
We define a representation $\gamma$ of $\Gamma$ by $*$-automorphisms of $\mathcal B(\Gamma)$ such that
$$
\gamma_g(f)(h) = f(g^{-1}h) \ .
$$
Since $\gamma_{g'}(1_g) = 1_{g'g}$ and $\gamma_{g}(\Delta_h) = \Delta_{hg^{-1}} - \Delta_{g^{-1}}$ it follows that $\gamma$ leaves $\mathcal B_{\omega}$ globally invariant and we can therefore also consider $\gamma$ as a representation of $\Gamma$ by $*$-automorphisms of $\mathcal B_{\omega}$. This induces a representation $\phi$ of $\Gamma$ by homeomorphisms of $\overline{\Gamma}^{\omega}$ defined such that
$$
\phi_g(\kappa) = \kappa \circ \gamma_{g^{-1}} \ .
$$
Since $\phi_g\left(\partial_{\omega} \Gamma\right) = \partial_{\omega} \Gamma$ we may also consider $\phi$ as a representation $\phi$ of $\Gamma$ by homeomorphisms of $\partial_{\omega} \Gamma$. The equality $\phi_g(c_h) = c_{gh}$ shows that the action $\phi$ is the unique action of $\Gamma$ on $\overline{\Gamma}^{\omega}$ which extends the canonical left action of $\Gamma$ on itself. As in the work of Rieffel we can then consider the crossed products $C(\partial_{\omega}\Gamma) \rtimes \Gamma$ and $C(\partial_{\omega}\Gamma) \rtimes_r \Gamma$.


For each $g \in \Gamma$ define $\overline{D}_g \in C\left(\overline{\Gamma}^{\omega}\right)$ such that
\begin{equation}\label{09-03-20}
\overline{D}_g(\kappa) \ = \ - \kappa(\Delta_g) \ .
\end{equation}
Let $D_g \in  C\left(\partial_{\omega} \Gamma\right)$ be the restriction of $\overline{D}_g$ to $ \partial_{\omega} \Gamma$. It straightforward to check that $D : \Gamma \ \to \ C(\partial_{\omega} \Gamma)$ and $\overline{D} : \Gamma \to C(\overline{\Gamma}^{\omega})$ are $\phi$-cocycle{\color{red}s} on $\partial_{\omega} \Gamma$ and $\overline{\Gamma}^{\omega}$, respectively. As explained in the previous section we obtain from $D$ in a natural way flows on both $C(\partial_{\omega}\Gamma) \rtimes \Gamma$ and $C(\partial_{\omega}\Gamma) \rtimes_r \Gamma$. It was shown by Rieffel in Section 3 of \cite{Ri} that a representation of the full crossed product $C(\partial_{\omega}\Gamma) \rtimes \Gamma$ is a copy of the cosphere algebra of a spectral triple associated to the pair $(\Gamma,\omega)$; a construction introduced by Connes, \cite{C}. When the action of $\Gamma$ on $\partial_{\omega}\Gamma$ is amenable, as it is when $\Gamma$ is amenable or hyperbolic, the representation is faithful and the cosphere bundle is then isomorphic to $C(\partial_{\omega}\Gamma) \rtimes \Gamma = C(\partial_{\omega}\Gamma) \rtimes_r \Gamma$ by Theorem 3.7 in \cite{Ri}. By comparing with Rieffels formula for the geodesic flow it is apparent that the flow $\alpha^D$ we consider here is the the geodesic flow on the cosphere algebra, provided $\omega \geq 0$ as will be the case in the following.


\section{Flows from random walks on $\Gamma$}\label{walks}
 Let $\mu : \Gamma \to [0,1]$ be a function such that $\sum_{g \in \Gamma} \mu(g) = 1$. We assume that the support $\supp \mu = \left\{ g \in \Gamma: \ \mu(g) > 0\right\}$ of $\mu$ generates $\Gamma$ as a semi-group. As outlined in the introduction the pair $(\Gamma,\mu)$ can be interpreted as a random walk on $\Gamma$ such that the probability of going from $x$ to $y$ is $\mu(x^{-1}y)$. We shall rely on many results from this area of mathematical research and refer to the two monographs \cite{Wo1} and \cite{Wo2} by W. Woess for an introduction.

  The function $\mu$ gives rise to a matrix over $\Gamma$ which we also denote by $\mu$, namely $\mu_{x,y} = \mu(x^{-1}y)$. The higher powers $\mu^n, n =0,1,2,\cdots$, of $\mu$ can then be defined in the usual way and, in particular,
  $$
  \mu^0_{x,y}  \ = \ \begin{cases} 0 \ , \ x \neq y , \\ 1 \ , \ x  =  y \ . \end{cases}
  $$  
  Thus the probability of ending at $y \in \Gamma$ after $n$ steps when starting at $x\in \Gamma$ is the number $\mu^n_{x,y}$. We will assume that the random walk is transient, meaning that
\begin{equation*}\label{transience}
G(x,y) \ \overset{def}{=} \ \sum_{n =0}^{\infty} \mu^n_{x,y} \ < \ \infty
\end{equation*}
for all $x,y \in \Gamma$. The function $G : \Gamma \times \Gamma \to \mathbb R$ is known as the Green kernel of the random walk $(\Gamma,\mu)$. Note that $G$ is $\Gamma$-invariant in the sense that
\begin{equation}\label{28-05-19b}
G(gx,gy) \ = \ G(x,y)
\end{equation}
for all $g,x,y \in \Gamma$.
Since $\supp \mu$ generates $\Gamma$ the Green kernel is positive and we define $\omega_{\mu} : \Gamma \to \mathbb R$ such that
$$
\omega_{\mu}(g) \ = \ \log G(e,e) \ - \ \log G(e,g)  \ .
$$
The number $\omega_{\mu}(g)$ is the distance from $e$ to $g$ in the Green metric which was introduced by Blach\`ere and Brofferio in \cite{BF}. The Green metric is not always a genuine metric.

\begin{lemma}\label{28-05-19a} 
$$
\left|\omega_{\mu}(gh) \ - \ \omega_{\mu}(h)\right| \ \leq \ \log G(e,e) \ - \ \min \left\{ \log G(e,g^{-1}), \ \log G(e,g)\right\}  
$$ 
for all $g,h \in \Gamma$.
\end{lemma}
\begin{proof} Since $G(y,y) = G(e,e)$ for all $y\in \Gamma$ it follows by combining (b) of Theorem 1.38 in \cite{Wo2} with Proposition 1.43 in \cite{Wo2} that
\begin{equation}\label{01-05-20}
\frac{G(e,g)}{G(e,e)} \frac{G(g,gh)}{G(e,e)} \ \leq \ \frac{G(e,gh)}{G(e,e)} 
\end{equation}
for all $g,h \in \Gamma$, which implies that $\omega_{\mu}(gh) - \omega_{\mu}(h) \ \leq \ \log G(e,e) - \log G(e,g)$. Thus
$\omega_{\mu}(h)  -  \omega_{\mu}(gh)  =  \omega_{\mu}(g^{-1}gh)  -  \omega_{\mu}(gh) \ \leq \ \log G(e,e)  -  \log G(e,g^{-1})$.
\end{proof}

It follows from Lemma \ref{28-05-19a} that $\omega_{\mu}$ is translation bounded and we can now construct the $\omega_{\mu}$-boundary as in Section \ref{bussec}. We denote it by $\partial_{\mu}\Gamma$. The corresponding functions from \eqref{28-05-19d} are then given by
\begin{align*}\label{29-05-19c}
&\Delta^{\mu}_g(h) \ \overset{def}{=} \ \omega_{\mu}(gh) - \omega_{\mu}(h) \ = \ \log G(e,h) - \log G(e,gh) \\
& = \ - \log \frac{G(g^{-1},h)}{G(e,h)} \ .
\end{align*}
The function $K: \Gamma \times \Gamma \to ]0,\infty[$ defined by
$$
K(g,h) \ = \ \frac{G(g,h)}{G(e,h)}
$$
is the Martin kernel, cf. \cite{Wo2}. We note that it follows from  Theorem 1.38 and Proposition 1.43 in \cite{Wo2} that
\begin{equation}\label{01-05-20a}
\frac{G(g,e)}{G(e,e)} \leq K(g,h) \leq \frac{G(e,e)}{G(e,g)} \ \ \forall g,h \in \Gamma \ ,
\end{equation}
and from \eqref{28-05-19b} that
\begin{equation}\label{01-05-20e}
K(g^{-1}, h)K(x,gh) \ = \ K(g^{-1}x,h) \ \ \forall g,h,x \in \Gamma \ .
\end{equation}
Since
\begin{equation}\label{29-05-19d}
\Delta^{\mu}_g(h) \ = \ - \log K(g^{-1},h) \ ,
\end{equation}
it follows from the construction of $\overline{\Gamma}^{\omega_{\mu}}$ that the function
$$
\Gamma \ni h \mapsto K(g,h) \ = \ \exp \left(-\Delta^{\mu}_{g^{-1}}(h)\right)
$$
extends to a continuous function on $\overline{\Gamma}^{\omega_{\mu}}$ for all $g$ and the resulting functions collectively separate the points of $\partial_{\mu} \Gamma$. It follows therefore that there is an identification 
\begin{equation}\label{29-05-19}
\partial_{\mu}\Gamma \ = \ \partial_M \Gamma \ ,
\end{equation}
where $\partial_{\mu}\Gamma$ is the $\omega_{\mu}$-boundary of Rieffel and $\partial_M \Gamma$ denotes the Martin boundary of the random walk $(\Gamma,\mu)$, cf. Definition 7.17 in \cite{Wo2}. The action $\phi$ of $\Gamma$ on $\partial_M \Gamma$ which results from the identification \eqref{29-05-19} is the same as the one usually considered on the Martin boundary; a fact which follows because both are extensions by continuity of the left action of $\Gamma$ on itself. From now on we mostly suppress $\phi$ in the notation and write $\phi_g(\xi) = g\xi$ when $\xi \in \partial_M \Gamma$.

It follows that we get a cocycle $D^{\mu} : \Gamma \to C(\partial_M \Gamma)$ from the functions $\Delta^{\mu}_g$, and by comparing \eqref{29-05-19d} and \eqref{09-03-20} we find that 
\begin{equation}\label{20-01-20}
D^{\mu}_g(\xi) \ = \ \log K(g^{-1},\xi) \ , \ \ \ \ \ \xi \in \partial_M \Gamma \ ,
\end{equation}
where $K : \Gamma \times \partial_M \Gamma \to ]0,\infty[$ is the continuous extension of the Martin kernel $K :\Gamma \times \Gamma \to ]0,\infty[$. We note that
$$
\frac{G(g,e)}{G(e,e)} \ \leq \ K(g,\xi) \ \leq \ \frac{G(e,e)}{G(e,g)} \ \ \ \  \forall g \in \Gamma \ \forall \xi \in \partial_M \Gamma \ 
$$ 
and
\begin{equation}\label{16-04-19x}
K(x,g\xi){K(g^{-1},\xi)} \ = \ {K(g^{-1}x,\xi)} \ \ \ \ \ \forall x,g \in \Gamma, \ \forall \xi \in \partial_M\Gamma  \ .
\end{equation}

In the following we seek to determine the KMS-states for the flow $\alpha^{D^{\mu}}$ on $C\left(\partial_M \Gamma\right) \rtimes_r \Gamma$. Since
$$
e^{\beta D^{\mu}_g} \ = \ K(g^{-1}, \cdot)^{\beta} \ ,
$$
an $e^{\beta D^{\mu}}$-conformal measure will in the following be called a $K^{\beta}$-conformal measure and $K$-conformal when $\beta =1$.

\section{The $K^{\beta}$-conformal measures}

By definition we have the following, cf. \eqref{14-03-19a} and \eqref{20-01-20}.
\begin{lemma}\label{24-04-19a} Let $\beta \in \mathbb R$. A Borel probability measure $m$ on $\partial_M \Gamma$ is $K^{\beta}$-conformal if and only if
\begin{equation}\label{12-01-20a}
m\left(g^{-1}  B\right) \ = \ \int_BK(g,\xi)^{\beta} \ \mathrm{d} m(\xi)
\end{equation}
for all $g \in \Gamma$ and every Borel set $B \subseteq \partial_M \Gamma$.
\end{lemma} 
As we now explain it follows from some of the fundamental results on random walks that there always exists a $K^{\beta}$-conformal measure on $\partial_M \Gamma$ when $\beta = 1$.

 The path space of the random walk is the space
$\Gamma^{\mathbb N} = \prod_{i=0}^{\infty}\Gamma$ consisting of sequences $x = (x_0,x_1,x_2,\cdots)$ where $x_i \in \Gamma$ for all $i$ which we consider as a topological space with the product topology. Note that $\Gamma$ acts on $\prod_{i=0}^{\infty}\Gamma$ by homeomorphism in the natural way: $gx = (gx_0,gx_1,gx_2, \cdots )$. Set 
$$
\mathcal C_0 \ = \ \left\{(x_n)_{n=1}^{\infty} : \  \lim_{n \to \infty} K(h,x_n) \ \text{exists for all} \ h \in \Gamma \ \right\} \ ;
$$
a Borel subset of $\Gamma^{\mathbb N}$. For $(x_n)_{n =1}^{\infty} \in \Gamma^{\mathbb N}$ we write $\lim_{n \to \infty} x_n = \infty$ when $(x_n)_{n=1}^{\infty}$ eventually stays out of every finite subset of $\Gamma$, and note that
$$
\mathcal C \ = \ \left\{ (x_n)_{n=1}^{\infty} \in \mathcal C_0: \ \lim_{n \to \infty} x_n = \infty \right\}
$$
is also a Borel subset of $\Gamma^{\mathbb N}$. It follows that we can define a Borel map $\chi : \mathcal C \to \partial_M\Gamma$ such that
$$
K(h,\chi(x)) = \lim_{n \to \infty} K(h, x_n) \ \ \forall h \in \Gamma \ 
$$
when $ x = (x_n)_{n=1}^{\infty} \in \mathcal C$. It follows from \eqref{16-04-19x} that $\mathcal C$ is $\Gamma$-invariant and $\chi$ equivariant; i.e. $g\mathcal C = \mathcal C$ and $\chi(gx) = g\chi(x)$. Note that $\chi$ is a surjective Borel map.


For each group element $g \in \Gamma$ there is a continuous map $\pi_g : \prod_{i=1}^{\infty} \Gamma \to \prod_{i=1}^{\infty} \Gamma$ defined such that
$$
\pi_g(s_1,s_2,s_3 {\color{red} , }\cdots) \ = \ (g,gs_1,gs_1s_2, gs_1s_2s_3,\cdots ) \ .
$$
Note that $g \pi_e(x) = \pi_g(x)$. Let $\mu^{\infty}$ be the Borel probability measure on $\prod_{i=1}^{\infty}\Gamma$ obtained as the infinite product of the measure $\mu$. The push-forward of $\mu^{\infty}$ under the map $\pi_g$ is the Borel probability measure ${}_g\mathbb P^{\mu}$ on $\prod_{i=1}^{\infty}\Gamma$ defined such that
$$
{}_g\mathbb P^{\mu}(B) \ = \ \mu^{\infty}\left(\pi_g^{-1}(B)\right)  \ .
$$
It will be important to note that
\begin{equation*}\label{23-04-19e}
{}_g\mathbb P^{\mu}(B) \ = \ {}_e\mathbb P^{\mu}\left( g^{-1} B\right) \ 
\end{equation*}
for every Borel subset $B \subseteq \prod_{i=1}^{\infty} \Gamma$. The measures ${}_g\mathbb P^{\mu}(B)$ are all concentrated on $\mathcal C$, cf. \cite{Wo2}, and we push ${}_g \mathbb P^{\mu}$ forward via $\chi$ to $\partial_M \Gamma$; that is, we define a Borel probability measure ${}_g\nu$ on $\partial_M \Gamma$ by 
$$
{}_g\nu(B) \ = \ {}_g\mathbb P^{\mu}\left(\chi^{-1}(B)\right) \ .
$$  
Since $\chi$ is equivariant and $\pi_g = g\pi_e$ it follows that
\begin{equation}\label{23-04-19f}
{}_g \nu(B) \ = \ {}_e\nu\left(g^{-1} B\right) \ .
\end{equation}

The minimal Martin boundary $\partial_m\Gamma$ is the set of elements $\xi \in \partial_M \Gamma$ with the property that the function $\Gamma \ni g \mapsto K(g,\xi)$ is a minimal $\mu$-harmonic function, i.e. is an extremal point in the convex set of function {\color{red} s} $\varphi : \Gamma \to [0,\infty)$ for which
\begin{itemize}
\item $\sum_{s \in \Gamma} \mu(s)\varphi(gs) \ = \ \varphi(g)$ for all $g \in \Gamma$, and
\item $\varphi(e) = 1$.
\end{itemize}
 The minimal Martin boundary $\partial_m \Gamma$ is a Borel subset of $\partial_M \Gamma$, cf. Lemma 7.52 in \cite{Wo2}. 
 
\begin{thm}\label{07-01-20xx} The measures ${}_g\nu$ are all concentrated on $\partial_m\Gamma$, i.e. ${}_g \nu(\partial_m \Gamma) =1$ for all $g \in \Gamma$.
\end{thm}
\begin{proof} See Section 7 in \cite{Wo2} or \cite{Sa}.
\end{proof} 
 
  The measure ${}_e\nu$ on $\partial_M \Gamma$ will be called the harmonic measure of the random walk $(\Gamma ,\mu)$. It is of fundamental importance that the measures ${}_g \nu$ are absolutely continuous with respect to each other and that the corresponding Radon-Nikodym derivatives are given by the Martin kernels. More precisely, by combining \eqref{23-04-19f} with Theorem 7.42 in \cite{Wo2} we get the following. Alternatively, see Theorem 5.1 in \cite{Sa}.

\begin{thm}\label{23-04-19c} For every $g \in \Gamma$ and every Borel subset $B \subseteq \partial_M \Gamma$,
$$
{}_e\nu(g^{-1} B) \ = \ {}_g\nu(B) \ = \ \int_B K(g,\xi) \ \mathrm{d}{}_e\nu(\xi) \ .
$$
\end{thm}

By comparing with Lemma \ref{24-04-19a} we get

\begin{cor}\label{25-04-19} The harmonic measure ${}_e\nu$ is $K$-conformal and ergodic for the action of $\Gamma$.
\end{cor}
\begin{proof} It follows immediately from Theorem \ref{23-04-19c} that ${}_e\nu$ is $K$-conformal. To see that it is also ergodic assume that $A \subseteq \partial_M \Gamma$ is a $\Gamma$-invariant Borel set such that ${}_e\nu(A) \neq 0$. Define a Borel probability measure $\nu$ on $\partial_M \Gamma$ such that $\nu(B) = {}_e\nu(A)^{-1} {}_e\nu(A \cap B)$. It follows from Theorem \ref{23-04-19c} and Theorem \ref{07-01-20xx} that
\begin{align*}
&\int_{\partial_m\Gamma} K(g,\xi) \ \mathrm{d}{}\nu(\xi) \ =  {}_e\nu(A)^{-1}  \ \int_{A} K(g,\xi) \ \mathrm{d}{}_e\nu(\xi) \\
& = \  {}_e\nu(A)^{-1}  {}_e\nu(g^{-1}A) \ = \ 1  
\end{align*}
for all $g$ and then from Theorem 7.53 in \cite{Wo2} that $\nu = {}_e \nu$, i.e. ${}_e\nu(A) = 1$. 
 \end{proof}
 
The harmonic measure ${}_e\nu$ plays a fundamental role in the theory of random walks, partly because the measure space $(\partial_M\Gamma,{}_e\nu)$ is a realization of the socalled Poisson boundary, cf. \cite{Wo2}. Note that this measure space is isomorphic to $(\partial_m\Gamma,{}_e\nu)$ by Theorem \ref{07-01-20xx}. It is known, \cite{Ka1}, that ${}_e\nu$ is either supported on a single point (i.e. the Poisson boundary is trivial) or else is non-atomic.

To proceed we need a condition which ensures that the function $\Gamma \ni g \mapsto K(g,\xi)$ is $\mu$-harmonic for every $\xi \in \partial_M \Gamma$. Let $d$ be the word metric on $\Gamma$ coming from a symmetric finite subset $S \subseteq \Gamma$ which generates $\Gamma$ as a semi-group. There is a constant $C > 0$ such that 
\begin{equation}\label{harnach}
\frac{G(x,z)}{G(y,z)} \ \leq \ C^{d(x,y)} 
\end{equation}
for all $x,y,z \in \Gamma$; this is known as a Harnack inequality, cf. (2.2) in \cite{Go}. We will assume that
\begin{equation}\label{07-01-20}
\sum_{g \in \Gamma} \mu(g) C^{d(g,e)} \ < \ \infty \ ;
\end{equation} 
 a condition trivially satisfied when $\mu$ has finite support, but also when $\mu$ has superexponential tail in the sense of \cite{Go}. With that assumption the following lemma can be found as Lemma 7.1 in \cite{GGPY}. We give here a more direct proof.

\begin{lemma}\label{30-05-19} Assume that \eqref{07-01-20} holds. It follows that 
\begin{equation*}\label{30-05-19a}
\sum_{s \in \Gamma} \mu(s) K(gs,\xi) \ = \ K(g,\xi)
\end{equation*}
for all $g \in \Gamma$ and all $\xi \in \partial_M \Gamma$.
\end{lemma}
\begin{proof} Fix $\xi \in \partial_M \Gamma$ and choose a sequence $\{x_n\}$ in $\Gamma$ such that $\lim_{n \to \infty} x_n = \xi$ in $\overline{\Gamma}^{\omega_{\mu}}$. Fix $g \in \Gamma$. Using \eqref{01-05-20e} we find that
\begin{align*}
&\sum_{s \in \Gamma} \mu(s) K(gs,\xi) \ = \ \sum_{s \in \Gamma} \mu(s) \lim_{n \to \infty} K(gs,x_n ) \\
& = \sum_{s \in \Gamma} \mu(s) \lim_{n \to \infty} K(s,g^{-1}x_n )K(g,x_n) \ .
\end{align*} 
From \eqref{01-05-20a} and \eqref{harnach} we get that
$$
K(s,g^{-1}x_n)K(g,x_n) \ \leq \  K(g,x_n)C^{d(s,e)} \ \leq \  \frac{G(e,e)}{G(e,g)}  C^{d(s,e)} \ .
$$
Thanks to \eqref{07-01-20} we can therefore use Lebesgues  dominated convergence theorem to conclude that
\begin{align*}
&\sum_{s \in \Gamma} \mu(s) \lim_{n \to \infty} K(s,g^{-1}x_n )K(g,x_n) \ =  \  \lim_{n \to \infty} \sum_{s \in \Gamma} \mu(s) K(s,g^{-1}x_n )K(g,x_n) \\
& = \ \lim_{n \to \infty} \sum_{s \in \Gamma} \mu(s) K(gs,x_n) \ .
\end{align*} 
Since
\begin{align*}
&\sum_{s \in \Gamma} \mu(s) K(gs,x_n) \ = \  \frac{1}{G(e,x_n)}\sum_{s \in \Gamma} \mu_{g,gs} \sum_{m=0}^{\infty} \mu^m_{gs,x_n} \\
& = \ \frac{1}{G(e,x_n)} \sum_{m=1}^{\infty} \mu^m_{g,x_n} \ = \ \begin{cases} K(g,x_n) \ , & \ x_n \neq g \\ K(x_n,x_n) - {G(e,x_n)}^{-1} \ , & \ x_n = g \ , \end{cases}
\end{align*}
it follows that $\lim_{n \to \infty} \sum_{s \in \Gamma} \mu(s) K(gs,x_n) = K(g,\xi)$ because $x_n \neq g$ for all large $n$.
\end{proof}

\begin{defn}\label{12-01-20} An element $\xi \in \partial_M\Gamma$ is called a \emph{spine} when $K(g,\xi) = 1$ for all $g \in \Gamma$.
\end{defn}

We have no good reason for this name, except for the observation that an element of the Martin boundary with the above property in \cite{CS} was compared to an object in potential theory called a 'Lebesgue spine'. 

\begin{lemma}\label{12-01-20b} A spine $\xi_{0} \in \partial_M\Gamma$ is fixed by $\Gamma$; viz. $g \xi_0 = \xi_0$ for all $g \in \Gamma$.
\end{lemma}
\begin{proof} It follows from \eqref{16-04-19x} that $K(h, g \xi_0) = 1$ for all $g,h \in \Gamma$ and hence that $g\xi_0 = \xi_0$ since the functions $K(h, \cdot), \ h \in \Gamma$, separate the points of $\partial_M \Gamma$. 
\end{proof}

\begin{prop}\label{07-01-20d} Assume that \eqref{07-01-20} holds and let $\beta \in \mathbb R$. A Borel probability measure $m$ on $\partial_M \Gamma$ is $K^{\beta}$-conformal if and only if one of the following three alternatives hold:
\begin{itemize}
\item[A)] $m$ is the Dirac measure $m = \delta_{\xi_0}$  concentrated on a spine $\xi_0 \in \partial_M \Gamma$. 
\item[B)] $\beta = 0$ and $m$ is $\Gamma$-invariant.
\item[C)] $\beta = 1$ and 
\begin{equation}\label{07-01-20f}
m(g^{-1} B) \ =  \ \int_{B} K(g,\xi) \ \mathrm{d} m(\xi)
\end{equation}
for all Borel subsets $B \subseteq \partial_M \Gamma$ and all $g \in \Gamma$.
\end{itemize}
 \end{prop}
 \begin{proof} The 'if' part is trivial so we prove the 'only if' part. Let $m$ be a $K^{\beta}$-conformal measure and assume that A) does not hold. In view of Lemma \ref{24-04-19a} it suffices to show that $\beta \in \{0,1\}$. For each $g \in \Gamma$ set
 $$
 P_g = \left\{ \xi \in \partial_M \Gamma: \ K(g,\xi) = 1 \right\} \ .
 $$
 If $m(P_g) = 1$ for all $g \in \Gamma$ it follows that $m\left(\bigcap_g P_g\right) = 1$. Since the functions $K(g, \cdot)$ separate the points of $\partial_M \Gamma$ it follows that $ \bigcap_g P_g$ contains exactly one point which must be a spine on which $m$ is concentrated, which is impossible since we assume that A) does not holds. It follows that $m(P_{g_0}) < 1$ for some $g_0 \in \Gamma$. This implies that there is a Borel set $A \subseteq \partial_M\Gamma$ of positive $m$-measure such that $K(g_0, \xi) \neq 1$ for all $\xi \in A$. Let $x,y \in \mathbb R$, $x \neq y$ and $t \in ]0,1[$. It follows that
$$
K(g_0,\xi)^{t x + (1-t)y} \ < \ t K(g_0,\xi)^x + (1-t) K(g_0,\xi)^y
$$
for all $\xi \in A$ and hence 
\begin{align*}
&\int_{\partial_M \Gamma} K(g_0,\xi)^{t x + (1-t)y} \ \mathrm{d} m(\xi) \ \\
&=  \ \int_{\partial_M \Gamma \backslash A} K(g_0,\xi)^{t x + (1-t)y} \ \mathrm{d} m(\xi) \ + \  \int_{A} K(g_0,\xi)^{t x + (1-t)y} \ \mathrm{d} m(\xi) \\
& \ <   \int_{\partial_M \Gamma \backslash A}t K(g_0,\xi)^x + (1-t) K(g_0,\xi)^y \ \mathrm{d} m(\xi) \\
& \ \ \ \ \ \ \ \ \ \ \ \ \ \ \ + \ \int_{ A}t K(g_0,\xi)^x + (1-t) K(g_0,\xi)^y  \  \mathrm{d} m(\xi)\\
& =  \int_{\partial_M \Gamma}t K(g_0,\xi)^x + (1-t) K(g_0,\xi)^y \  \mathrm{d} m(\xi) \ .
\end{align*}
This shows that the function
$$
t \ \mapsto \ \int_{\partial_M \Gamma} K(g_0,\xi)^{t} \ \mathrm{d} m(\xi)
$$
is strictly convex on $\mathbb R$. Choose $n$ so large that $\mu^n_{e,g_0} > 0$. Such an $n$ exists because the support of $\mu$ generates $\Gamma$ by assumption. The function
$$
\Phi(t) \ = \ \sum_{h \in \Gamma} \mu^n_{e,h} \int_{\partial_M \Gamma} K(h,\xi)^{t} \ \mathrm{d} m(\xi)
$$
is then also strictly convex. Since $m$ is $K^{\beta}$-conformal, 
$$
\int_{\partial_M \Gamma} K(h,\xi)^{\beta} \ \mathrm{d} m(\xi)  \ = \ 1
$$
for all $h \in \Gamma$ by \eqref{21-01-20b}. As $\sum_{h \in \Gamma} \mu^n_{e,h} = 1$ we find therefore that 
$$
\Phi(\beta) =  \sum_{h \in \Gamma} \mu^n_{e,h} \int_{\partial_M \Gamma} K(h,\xi)^{\beta} \ \mathrm{d} m(\xi) \ = \sum_{h \in \Gamma} \mu^n_{e,h}  \ = \ 1 \ .
$$
For a similar reason $\Phi(0) = 1$. Iterated applications of Lemma \ref{30-05-19} show that 
\begin{align*}
&\Phi(1) =  \int_{\partial_M \Gamma} \sum_{h \in \Gamma} \mu^n_{e,h}  K(h,\xi) \ \mathrm{d} m(\xi)  = \ \int_{\partial_M \Gamma} K(e,\xi) \ \mathrm{d} m(\xi) \\
& = \  \int_{\partial_M \Gamma} 1 \ \mathrm{d} m(\xi)\ = \ 1 \ .
\end{align*}
Thus $\Phi(0) = \Phi( \beta) = \Phi(1) = 1$. Since $\Phi$ is strictly convex this implies that $\beta \in \{0,1\}$ which is what we needed to prove.

\end{proof}




\begin{cor}\label{24-01-20}  Assume that \eqref{07-01-20} holds. The KMS spectrum of $\alpha^{D^{\mu}}$ is $\mathbb R$ when $\partial_M \Gamma$ contains a spine, and $\{0,1\}$ or $\{1\}$ when it does not. 
\end{cor}

\section{Trivial Poisson boundary}

\begin{prop}\label{24-01-20a} Let $(\Gamma,\mu)$ be a random walk on $\Gamma$ for which \eqref{07-01-20} holds. Assume that the Martin boundary $\partial_M \Gamma$ contains a spine $\xi_{0}$.  For every $\beta \notin \{0,1\}$ there is an affine homeomorphism 
$\tau \ \mapsto \ \omega_{\tau}$ 
from the tracial state space $T(C^*_r(\Gamma))$ of $C^*_r(\Gamma)$ onto the simplex of $\beta$-KMS states for the flow $\alpha^{D^{\mu}}$ on $C\left(\partial_M \Gamma\right)\rtimes_r \Gamma$ such that
\begin{equation}\label{03-06-19bc}
\omega_{\tau}(fU_g) \ = \ f(\xi_0)\tau(U_g) 
\end{equation}
for all $f\in C(\partial_M \Gamma)$ and all $g \in \Gamma$.
\end{prop}
\begin{proof} Let $\tau \in T(C^*_r(\Gamma))$. Let $R(\Gamma)$ be the amenable radical of $\Gamma$, i.e. $R(\Gamma)$ is the largest amenable normal subgroup of $\Gamma$. Then $C(\partial_M \Gamma) \rtimes R(\Gamma) = C(\partial_M \Gamma) \rtimes_r R(\Gamma)$ since $R(\Gamma)$ is amenable. By considering $ \left(\delta_{\xi_0}, \tau|_{C^*(R(\Gamma))}\right)$ as a $\delta_{\xi_0}$-measurable field of traces it follows from Theorem 2.1 in \cite{N} that there is a trace state $\omega'$ on $C(\partial_M \Gamma) \rtimes_r R(\Gamma)$ such that
$$
\omega'(f U_g) \ = \ f(\xi_0)\tau(U_g) \  
$$
for all $f \in C(\partial_M\Gamma)$ and all $g \in R(\Gamma)$. By Theorem 4.1 in \cite{BKKO} $\tau(U_g) = 0$ for all $g \notin R(\Gamma)$, so when we set $\omega_{\tau} = \omega' \circ p$ where  $p : C\left(\partial_M \Gamma\right)\rtimes_r \Gamma \to C\left(\partial_M \Gamma\right)\rtimes_r R(\Gamma)$ is the canonical conditional expectation we obtain a state $\omega_{\tau}$ such that \eqref{03-06-19bc} holds. It follows from \eqref{03-06-19bc} that\begin{align*}
& \omega_{\tau}\left(fU_gf_1U_{g_1}\right) \ = \ \omega_{\tau} \left(ff_1 \circ \phi_{g^{-1}} U_gU_{g_1}\right) \ = \ f(\xi_0)f_1(g^{-1}\xi_0) \tau\left(U_gU_{g_1}\right) \ 
\end{align*}
and
\begin{align*}
& \omega_{\tau}\left(f_1U_{g_1}\alpha^{D^{\mu}}_{i\beta}(fU_g)\right) \ = \ \omega_{\tau}\left(f_1U_{g_1}f U_g e^{\beta D^{\mu}_g}\right) \\ 
& = \  \omega_{\tau}\left( f_1 f\circ \phi^{-1}_{g_1}  e^{\beta D^{\mu}_g \circ \phi^{-1}_{g}\circ \phi^{-1}_{g_1}} U_{g_1}U_g\right) \\
& = \ f_1(\xi_0) f(g_1^{-1}\xi_0) K( g^{-1}, g^{-1} g_1^{-1} \xi_0)^{\beta} \tau(U_{g_1}U_g) \ 
\end{align*}
Therefore the fact that $\xi_0$ is a spine, and in particular fixed by $\Gamma$, and $\tau$ is a trace imply that $\omega_{\tau}$ is a $\beta$-KMS state for $\alpha^{D^{\mu}}$. Hence the map of $\tau \mapsto \omega_{\tau}$ is well-defined. It is obviously injective so it remains only to show that it is also surjective. Consider therefore a $\beta$-KMS state $\omega$ for $\alpha^{D^{\mu}}$. By composing $\omega$ with the canonical surjection $q: C(\partial_M\Gamma) \rtimes \Gamma \to C(\partial_M\Gamma) \rtimes_r \Gamma$ we obtain the $\beta$-KMS state $\omega \circ q$ on $C(\partial_M \Gamma) \rtimes \Gamma$ for which we can apply \cite{N}.  Since we assume that $\beta \notin \{0,1\}$ it follows from Proposition \ref{07-01-20d} that every $K^{\beta}$-conformal measure is concentrated on $\xi_0$. Hence Theorem 1.3 in \cite{N} tells us that there is a trace state $\tau'$ on $C^*( \Gamma)$ such that
\begin{equation}\label{03-06-19ac}
\omega \circ q(fU_g) \ = \ f(\xi_0)\tau'(U_g)
\end{equation}
for all $f \in C(\partial_M\Gamma)$ and all $g \in \Gamma$. Since $\xi_0$ is fixed by $\Gamma$ every representation of $\Gamma$ extends to a representation of $ C(\partial_M\Gamma) \rtimes \Gamma$ implying that $C^*(\Gamma) \subseteq  C(\partial_M \Gamma) \rtimes \Gamma$. The restriction of $q$ to $C^*(\Gamma)$ is the canonical surjection $C^*(\Gamma) \to C^*_r(\Gamma)$ and it follows therefore from \eqref{03-06-19ac} that $\tau' $ factorises through $C^*_r(\Gamma)$. This shows that the map under consideration is also surjective.  
\end{proof}

As examples in Section \ref{09-03-20c} will show there can in general be many $K$-conformal measures when the Martin boundary contains a spine, but as we show next this requires that the Poisson boundary is non-trivial.

 A Borel probability measure $m$ on $\partial_M \Gamma$ is \emph{$\mu$-stationary} when
$$
\sum_{g \in \Gamma} \mu(g)m(g^{-1}B) = m(B)
$$
for all Borel sets $B \subseteq \partial_M \Gamma$.

\begin{lemma}\label{23-05-19ax} Assume that \eqref{07-01-20} holds. The following conditions are equivalent.
\begin{itemize}
\item[1)] The Poisson boundary of $(\Gamma,\mu)$ is trivial.
\item[2)] Every $\mu$-stationary Borel measure on $\partial_M \Gamma$ is $\Gamma$-invariant. 
\item[3)] Every $K$-conformal measure is $\Gamma$-invariant.
\item[4)] There is a point $\xi_0 \in \partial_m \Gamma$ such that the Dirac measure $\delta_{\xi_0}$ is the only $e^{\beta D^{\mu}}$-conformal measure on $\partial_M \Gamma$ for all $\beta \neq 0$.
\item[5)]  The minimal Martin boundary $\partial_m \Gamma$ contains a spine. 
\end{itemize}

\end{lemma}
\begin{proof}
1) $\Rightarrow$ 2): Let $m$ be a $\mu$-stationary Borel measure and consider a Borel set $B$ such that $m(B) \neq 0$. Since
$$
\sum_{h \in \Gamma} \mu(h) m\left(h^{-1}g^{-1} B\right) \ = \ m(g^{-1}B) \ ,
$$
the function $\Gamma \ni g \mapsto \frac{m(g^{-1}B)}{m(B)}$ is $\mu$-harmonic. Since it is also bounded the Poisson integral formula gives a bounded Borel function $\varphi_B$ on $\partial_M \Gamma$ such that
$$
m(g^{-1}B) \ = \ m(B) \int_{\partial_M \Gamma} K(g,\xi)\varphi_B(\xi) \ \mathrm{d} {}_e\nu(\xi) \ 
$$
for all $g \in \Gamma$, cf. Theorem 7.61 in \cite{Wo2}. That the Poisson boundary is trivial means that the measure ${}_e\nu$ is concentrated on a spine $\xi_0$. Therefore
$$
m(g^{-1}B) \ = \ m(B)  \varphi_B(\xi_0) \  
$$
for all $g \in \Gamma$. By taking $g =e$ we find that $\varphi_B(\xi_0) = 1$. This shows that $m(g^{-1}B) = m(B)$ for all $g \in \Gamma$ when $m(B) > 0$. The same is true when $m(B)=0$ because if $m(g^{-1}B) > 0$ it follows that $m(B) = m(gg^{-1}B) = m(g^{-1}B)$, a contradiction. 2) $\Rightarrow$ 3): Let $m$ be a $K$-conformal measure, and let $B \subseteq \partial_M\Gamma$ a Borel set. Using Lemma \ref{30-05-19} for the second identity we find that
\begin{align*}
&\sum_{g \in \Gamma} \mu(g) m(g^{-1}B) \ = \ \sum_{g \in \Gamma} \mu(g) \int_{B} K(g,\xi) \ \mathrm{d}m(\xi)\\
& = \ \int_B K(e,\xi) \ \mathrm{d} m(\xi)\  = \ m(B) \ ,
\end{align*}
proving that $m$ is $\mu$-stationary and hence $\Gamma$-invariant since we assume 2). 3) $\Rightarrow$ 4): Let $m$ be a $K$-conformal measure; there is one since the harmonic measure ${}_e\nu$ is $K$-conformal. Since $m$ is $\Gamma$-invariant it follows from \eqref{12-01-20a} that
$$
m(B) = \int_B K(g,\xi) \ \mathrm{d} m(\xi) 
$$
for all Borel sets $B \subseteq \partial_M \Gamma$ and all $g \in \Gamma$. This implies that $m$ is concentrated on a spine $\xi_0$. Note that $\xi_0 \in \partial_m\Gamma$ since ${}_e\nu$ is concentrated on $\partial_m\Gamma$. Hence 4) follows from Proposition \ref{07-01-20d} because a spine is unique. 4) $\Rightarrow$ 5): Since $\delta_{\xi_0}$ is $K$-conformal it follows that
$$
1 \ = \ \int_{\partial_M\Gamma} K(g^{-1},\xi) \ \mathrm{d}\delta_{\xi_0} = K(g^{-1},\xi_0)
$$
for all $g \in \Gamma$; i.e. $\xi_0$ is a spine and hence 5) holds. 5) $\Rightarrow$ 1): Let $\xi_0$ be a spine in $\partial_m\Gamma$. Then $\delta_{\xi_0}$ is a Borel probability measure concentrated on $\partial_m\Gamma$ representing the harmonic function $1$, and hence $\delta_{\xi_0} =  {}_e\nu$ by Theorem 7.53 in \cite{Wo2}, i.e. the Poisson boundary is trivial.
\end{proof}


\begin{thm}\label{03-06-19} Let $(\Gamma,\mu)$ be a random walk on $\Gamma$ for which \eqref{07-01-20} holds. Assume that the Poisson boundary of $(\Gamma,\mu)$ is trivial and let $\xi_0 \in \partial_m \Gamma$ be the point supporting the harmonic measure. For every $\beta \neq 0$ there is an affine homeomorphism 
$\tau \ \mapsto \ \omega_{\tau}$ 
from the tracial state space $T(C^*_r(\Gamma))$ of $C^*_r(\Gamma)$ onto the simplex of $\beta$-KMS states for the flow $\alpha^{D^{\mu}}$ on $C\left(\partial_M \Gamma\right)\rtimes_r \Gamma$ such that
\begin{equation*}\label{03-06-19b}
\omega_{\tau}(fU_g) \ = \ f(\xi_0)\tau(U_g) 
\end{equation*}
for all $f\in C(\partial_M \Gamma)$ and all $g \in \Gamma$.
\end{thm}
\begin{proof} By Proposition \ref{24-01-20a} it suffices here to handle the case $\beta =1$. But by 1) $\Rightarrow$ 4) of Lemma \ref{23-05-19ax} the Dirac measure $\delta_{\xi_0}$ is the only $K$-conformal measure and hence the arguments of Proposition \ref{24-01-20a} are valid also when $\beta =1$. 
\end{proof}


Much work has been concerned with the question about triviality or non-triviality of the Poisson boundary for random walks on groups; see \cite{KV} and \cite{BE} for some of the first and some of the most recent results, respectively. In particular, it follows from the Corollary to Theorem 4.2 in \cite{KV} that the Poisson boundary is never trivial for a random walk on a non-amenable group. Hence Theorem \ref{03-06-19} gives no information when $\Gamma$ is not amenable. In contrast, when $\Gamma$ is a finitely generated group of polynomial growth the Poisson boundary is trivial for all random walks, cf. page 466 in \cite{KV}. Hence Theorem \ref{03-06-19} applies to all groups of polynomial growth, including in particular all nilpotent groups.

\section{Non-elementary hyperbolic groups}

Beyond the nilpotent groups it seems that in relation to random walks the hyperbolic or word-hyperbolic groups of Gromow are the most studied and best understood. For a quick introduction to hyperbolic groups we refer to the survey by Kapovich and Benakli, \cite{KB}. When $\Gamma$ is a non-elementary hyperbolic group and $(\Gamma , \mu)$ is a random walk on $\Gamma$ with finite support (i.e. $\{g \in \Gamma: \ \mu(g) > 0 \}$ is finite) it was shown by Ancona in \cite{A} that the Martin boundary is $\Gamma$-equivariantly homeomorphic to the Gromov boundary $\partial \Gamma$ of $\Gamma$. 
We can then utilize some of the many results available on the action of $\Gamma$ on its Gromow boundary. The most important is the following, a result of Woess, \cite{Wo3}; see also Theorem 7.6 in \cite{Ka2}.

\begin{thm}\label{03-06-19d} Let $(\Gamma,\mu)$ be a random walk with finite support on the non-elementary hyperbolic group $\Gamma$. There is a unique $\mu$-stationary Borel probability measure on the Gromow boundary $\partial \Gamma$.
\end{thm}

Combined with the observation made in the proof of 2) $\Rightarrow$ 3) of Lemma \ref{23-05-19ax}, that all $K$-conformal measures are $\mu$-stationary, we only have to chase the litterature to obtain the following.

\begin{thm}\label{03-06-19e} Let $\Gamma$ be a torsion free non-elementary hyperbolic group and $(\Gamma,\mu)$ a random walk on $\Gamma$ with finite support. There is a $\beta$-KMS state for $\alpha^{D^{\mu}}$ if and only if $\beta = 1$. The $1$-KMS  state $\omega$ for $\alpha^{D^{\mu}}$ is unique and given by the formula
\begin{equation}\label{12-06-19}
\omega(a) = \int_{\partial_M \Gamma} Q(a) \ \mathrm{d}{}_e\nu
\end{equation}
where $Q : C(\partial_M \Gamma) \rtimes_r \Gamma \to C(\partial_M \Gamma)$ is the canonical conditional expectation and ${}_e \nu$ is the $\mu$-harmonic measure on $\partial_M \Gamma$.  
\end{thm}
\begin{proof} By Anconas result, \cite{A}, we can identify $\partial_M\Gamma$ and the Gromow boundary $\partial \Gamma$. The action of $\Gamma$ on $\partial \Gamma$ is minimal, cf. e.g. Proposition 4.2 in \cite{KB}, and it follows therefore that A) of Proposition \ref{07-01-20d} does not occur and hence the only values of $\beta$ for which there can be an $K^{\beta}$-conformal measure is $0$ and $1$. In particular, an $K^{\beta}$-conformal measure is $\mu$-stationary and it follows from Theorem \ref{03-06-19d} that the harmonic measure is the unique $K^{\beta}$-conformal measure. To see that the harmonic measure only is responsible for one $1$-KMS state for $\alpha^{D^{\mu}}$, note that it follows from Proposition 4.2 in \cite{KB} that every element of $\Gamma  \backslash \{e\}$ fixes exactly two points in $\partial_M \Gamma$. This implies, in particular, that the elements of $\partial_M \Gamma$ with non-trivial isotropy group are countable. In \cite{Wo3} Woess has shown that ${}_e\nu$ has no atoms, and it follows from this that ${}_e\nu$ is concentrated on points with no isotropy and hence from \cite{N} that there is only one $1$-KMS state, given by the formula \eqref{12-06-19}.
\end{proof}

As we have used above, in the setting of Theorem \ref{03-06-19e} the crossed product $C(\partial_M \Gamma) \rtimes_r \Gamma$ is independent of the random walk and agrees with $C(\partial \Gamma) \rtimes_r \Gamma$ thanks to the work of Ancona, \cite{A}. Theorem \ref{03-06-19e} can therefore be considered as a statement about different flows on the same $C^*$-algebra; $C(\partial \Gamma) \rtimes_r \Gamma$. It follows from the work of many hands, culminating in the work by Tu in \cite{Tu}, that $C(\partial \Gamma) \rtimes_r \Gamma$ is a Kirchberg algebra satisfying the UCT when $\Gamma$ has the properties specified in Theorem \ref{03-06-19e}. In general the structure of $C(\partial \Gamma) \rtimes_r \Gamma$ is unknown unless its K-theory groups are known, but for free groups the structure was completely unravelled by Spielberg in \cite{Sp} and his approach has been adopted to free products of finite groups over a common subgroup by Okayasu in \cite{O}. 
 
It should be noted that it is crucial in Theorem \ref{03-06-19e} that $\Gamma$ is torsion free. If for example $\Gamma$ is a direct product $\Gamma = H \times \mathbb F_2$ where $H$ is a finite group and $\mathbb F_2$ is the free group on two generators, then $\Gamma$ is non-elementary hyperbolic and for any finitely supported random walk $\mu$ on $\Gamma$ the harmonic measure is the only $K^{\beta}$-conformal measure. But there are many $1$-KMS states in this case; it is not difficult to see that the simplex of $1$-KMS states for the flow $\alpha^{D^{\mu}}$ is affinely homeomorphic to the tracial state space of $C^*(H)$. These additional $1$-KMS states arise because the action of $\Gamma$ on $\partial_M \Gamma$ is not free; the subgroup $H$ acts trivially on $\partial \Gamma$.

It would be wrong to conclude from Theorem \ref{03-06-19e} that the structure of KMS states for $\alpha^{D^{\mu}}$ is $\mu$-independent in the setting of Theorem \ref{03-06-19e}, although it is necessary to consider invariants finer than the set of inverse temperatures and the simplices of KMS-states. Specifically, it follows from \cite{INO} that the factor type of the harmonic measure varies with $\mu$, e.g. for nearest neighbour random walks on free groups.

\section{Examples with many $K$-conformal measures}\label{09-03-20c}

Fix a natural number $q \geq 2$ and let $\mathbb D = \mathbb Z_q \wr \mathbb Z $ be the wreath product of $\mathbb Z_q$ and $\mathbb Z$. In more detail $\mathbb D$ is the semi-direct product $\left( \oplus_{\mathbb Z} \mathbb Z_q\right) \rtimes \mathbb Z$ where $\mathbb Z$ acts on $  \oplus_{\mathbb Z} \mathbb Z_q$ by shifts, cf. e.g. \cite{BW}.

\begin{prop}\label{12-01-20d} (Woess, Brofferio-Woess) There is a finitely supported random walk $(\mathbb D,\mu_0)$ such that $\partial_M \mathbb D$ contains a spine in $\overline{\partial_m \mathbb D}\backslash \partial_m\mathbb D$.
\end{prop}
\begin{proof} Let $\mu_0$ be the random walk labelled (2.5) on page 420 of \cite{Wo4} for some $\alpha \neq 1/2$. As noted in remark (6) on page 432 of \cite{Wo4} the Martin boundary contains then a spine outside the minimal Martin boundary. That the spine is an element of the closure of the minimal Martin boundary follows from the description of the Martin kernels given in Theorem 5.7 of \cite{BW}. Specifically, in the notation of \cite{BW} it follows from Theorem 5.7 of \cite{BW}that the point $\omega_1^{ \infty}\omega_2^{-\infty}$ is a spine in $\partial_M \mathbb D \backslash \partial_m\mathbb D$. To see that it lies in the closure $\overline{\partial_m \mathbb D}$ of the minimal Martin boundary we use the notation from \cite{BW} and consider a sequence $\{\xi_n\}$ in $\partial^* \mathbb T_q$ such that $\lim_{n \to \infty} d(o_1,\xi_n \wedge \omega_1) = \infty$. It follows from Proposition 3.6 in \cite{BW} that $\lim_{n \to \infty} K_1(x_1,\xi_n) = 1$ for all $x_1 \in \mathbb T_q$ and then from Theorem 5.7 in \cite{BW} that $\lim_{n \to \infty} \xi_n \omega_2^{-\infty} = \omega_1^{\infty} \omega_2^{-\infty}$ in $\partial_M \Gamma$ while $\xi_n\omega_2^{-\infty} \in \partial_m \mathbb D$ for all $n$. Hence $\omega_1^{ \infty}\omega_2^{-\infty} \in \overline{\partial_m \mathbb D}\backslash \partial_m \mathbb D$. Similarly, $\omega_1^{- \infty}\omega_2^{\infty}$ is a spine in $\overline{\partial_m \mathbb D}\backslash \partial_m\mathbb D$ when $\alpha > \frac{1}{2}$. 
\end{proof}

Let $\xi_0$ be the spine in $\overline{\partial_m \mathbb D}\backslash \partial_m\mathbb D$. It follows that for all $\beta \notin \{0,1 \}$ the Dirac measure $\delta_{\xi_0}$ is the unique $K^{\beta}$-conformal measure on $\partial_M \mathbb D$ while for $\beta = 1$ there are at least two ergodic $K$-conformal measures, namely $\delta_{\xi_0}$ and the harmonic measure ${}_e \nu$; the latter is non-atomic by results of Kaimanovich. By going more into the details of \cite{BW} it is possible to show that $\delta_{\xi_0}$ and ${}_e\nu$ are the only ergodic $K$-conformal measures, but we skip the proof because we loose track of the exact number of $K$-conformal measures in the following constructions anyway.


Let $(\Gamma_1,\mu_1)$ and $(\Gamma_0,\mu_0)$ be finitely supported random walks such that $\overline{\partial_m \Gamma_0} \backslash \partial_m \Gamma_0$ contains a spine, e.g. the random walk from Proposition \ref{12-01-20d}. Let $0 < a < 1$ and define a probability measure $\mu_2$ on $\Gamma_2 = \Gamma_0 \times \Gamma_1$ such that
$$
\mu_2(g,h) = \begin{cases}a\mu_0(e)+(1-a)\mu_1(e) \ & \ \text{when} \ (g,h) = (e,e) \\
 a\mu_0(g) \ & \ \text{when} \ h = e \text{ and } g\neq e\\ (1-a)\mu_1(h) \ & \ \text{when} \ g = e \text{ and } h\neq e\\ 0 \ & \ \text{otherwise} \ . \end{cases}
$$

\begin{lemma}\label{12-01-20e} There is a continuous injective map $\Phi : \overline{\partial_m\Gamma_1} \to \overline{\partial_m \Gamma_2}$ defined such that 
\begin{equation}\label{13-01-20}
K((g,h),\Phi(\xi)) \ = \ K(h,\xi) 
\end{equation}
for all $(g,h) \in \Gamma_2$. 
\end{lemma}
\begin{proof} Let $\xi_0$ be the spine in $\overline{\partial_m \Gamma_0} \backslash \partial_m \Gamma_0$ and choose a sequence $\{\xi_n\}$ in $\partial_m \Gamma_0$ such that $\lim_{n \to \infty} K(g,\xi_n) = 1$ for all $g \in \Gamma_0$. Let $\xi \in \overline{\partial_m \Gamma_1}$ and choose a sequence $\{\eta_n\}$ in $\partial_m \Gamma_1$ such that $\lim_{n \to \infty} K(h,\eta_n) = K(h,\xi)$ for all $h \in \Gamma_1$. It follows from Theorem 3.3 in \cite{PW} that there is a sequence $\{\omega_n\}$ in $\partial_m \Gamma_2$ such that
$$
K((g,h), \omega_n) \ = \ K(g,\xi_n)K(h,\eta_n) \ \ \forall (g,h) \in  \Gamma_{2} \ .
$$
Note that $\{\omega_n\}$ converges in $\partial_M \Gamma_2$ to an element $\Phi(\xi) \in \overline{\partial_m \Gamma_2}$ for which \eqref{13-01-20} holds. The resulting map $\Phi : \overline{\partial_m\Gamma_1} \to \overline{\partial_m \Gamma_2}$ is clearly injective and continuous.
\end{proof}

\begin{lemma}\label{13-01-20b} The map $\Phi : \overline{\partial_m\Gamma_1} \to \partial_M\Gamma_2$ of Lemma \ref{12-01-20e} has the property that $(g,h)\Phi(\xi) = \Phi(h\xi)$ for all $(g,h) \in \Gamma_2$ and all $\xi \in \overline{\partial_m\Gamma_1}$.
\end{lemma}
\begin{proof} Using \eqref{16-04-19x} we find that
\begin{align*}
& K((g',h'), \Phi(h\xi)) \ = \ K(h',h\xi) \ = \ \frac{K(h^{-1}h',\xi)}{K(h^{-1},\xi)} \\
& =  \frac{K((g^{-1}g',h^{-1}h'),\Phi(\xi))}{K((g^{-1},h^{-1}),\Phi(\xi))} \ = \  K( (g',h'),(g,h)\Phi(\xi))  \ ,
\end{align*}
from which the statement follows.
\end{proof} 

\begin{lemma}\label{13-01-20c} $\Phi\left(\overline{\partial_m \Gamma_1}\right) \cap \partial_m \Gamma_2 = \emptyset$.
\end{lemma}
\begin{proof} Assume for a contradiction that there are points $\xi' \in \overline{\partial_m \Gamma_1}$ and $\xi \in \partial_m \Gamma_2$ such that $\Phi(\xi') = \xi$. It follows then from Theorem 3.3 in \cite{PW} that there are $s,t \in \mathbb R^+$ such that $as+(1-a)t = 1$ and non-negative functions $f : \Gamma_0  \to [0,\infty)$ and $f': \Gamma_1 \to [0,\infty)$ such that
\begin{enumerate}
\item[$\bullet$] $\sum_{x \in \Gamma_{0}} \mu_0(x)f(gx) = s f(g) \ \ \forall g \in \Gamma_0$,
\item[$\bullet$] $f$ is minimal with this property,
\item[$\bullet$] $\sum_{y \in \Gamma_1}\mu_{1}(y) f'(hy) = tf'(h) \ \ \forall h \in \Gamma_1 $,
\item[$\bullet$] $f'$ is minimal with this property, and
\item[$\bullet$] $K((g,h),\xi) \ = \ f(g)f'(h) \ \ \forall (g,h) \in \Gamma_{2}$.
\end{enumerate} 
Since $ K(h,\xi')  = K((g,h),\xi) = f(g)f'(h)$ for all $(g,h)$ it follows that $f$ is constant, say $f = c$ for some $c > 0$, and we find that
$$
c = \sum_{x \in \Gamma_0} \mu_0(x) c =  \sum_{x \in \Gamma_0} \mu_0(x) f(gx) = sf(g) = sc \ ,
$$
implying that $s = 1$. Hence $f =c$ is a minimal harmonic function on $\Gamma_0$, implying that the constant function $1$ is a minimal harmonic function on $\Gamma_0$. This means that the Poisson boundary of $(\Gamma_0,\mu_0)$ is trivial by Exercise 7.66 in \cite{Wo2}, and hence $\partial_m\Gamma_0$ contains a spine by Lemma \ref{23-05-19ax}. This contradicts that $\xi_0 \notin \partial_m \Gamma_0$.
\end{proof}

\begin{lemma}\label{13-01-20d} Let $\mu$ be an $K$-conformal measure concentrated on $\overline{\partial_m\Gamma_1}$. Then 
\begin{itemize}
\item $\mu \circ \Phi^{-1}$ is an $K$-conformal measure concentrated on $\overline{\partial_m \Gamma_2} \backslash \partial_m\Gamma_2$, 
\item $\mu \circ \Phi^{-1}$ is $\Gamma_2$-ergodic when $\mu$ is $\Gamma_1$-ergodic, 
\item $\mu \circ \Phi^{-1}$ is non-atomic when $\mu$ is, and
\item the map $\mu \mapsto \mu \circ \Phi^{-1}$ is injective.
\end{itemize}
\end{lemma}
\begin{proof} For the first item note that $\mu \circ \Phi^{-1}$ is concentrated on  $\overline{\partial_m \Gamma_2} \backslash \partial_m\Gamma_2$ by Lemma \ref{13-01-20c}. Let $B$ be a Borel subset of $\partial_M\Gamma_2$. We find that
\begin{align*}
& \mu \circ \Phi^{-1}\left((g,h)^{-1}B\right) \ = \ \mu \left(h^{-1}\Phi^{-1}(B)\right) \ \ \ \ \ \ \ \text{(by Lemma \ref{13-01-20b})}\\
& = \ \int_{\Phi^{-1}(B)} K(h,\xi) \ \mathrm{d} \mu(\xi) \\
& = \ \int_{\partial_M \Gamma_1} 1_B \circ \Phi(\xi) K((g,h),\Phi(\xi)) \ \mathrm{d}\mu(\xi) \\
& = \ \int_B K((g,h),\eta) \ \mathrm{d} \mu \circ \Phi^{-1}(\eta)  \ ,
\end{align*}
showing that $\mu \circ \Phi^{-1}$ is $K$-conformal. For the second item assume that $\mu$ is $\Gamma_1$-ergodic and let $B \subseteq \partial_M \Gamma_2$ be a $\Gamma_2$-invariant Borel set. It follows from Lemma \ref{13-01-20b} that $\Phi^{-1}(B)$ is $\Gamma_1$-invariant and hence that $\mu \circ \Phi^{-1}(B) \in \{0,1\}$. The third and fourth items follow from the injectivity of $\Phi$.
\end{proof}  

It follows from Lemma \ref{13-01-20d} that if there are $n$ $K$-conformal and ergodic measures concentrated on $\overline{\partial_m \Gamma_1}$, there will be at least $n+1$ $K$-conformal and ergodic measures concentrated on $\overline{\partial_m \Gamma_2}$ since the harmonic measure of $(\Gamma_2,\mu_2)$ will contribute an ergodic $K$-conformal measure singular to those coming from $(\Gamma_1,\mu_1)$. If we assume that there is a spine $\xi_1$ in $\overline{\partial_m \Gamma_1}$ the element $\Phi(\xi_1)$ will be a spine in $\overline{\partial_m \Gamma_2} \backslash \partial_m \Gamma_2$ and the harmonic measure of $(\Gamma_2,\mu_2)$ will be non-atomic by the result of Kaimanovich. Under the same assumption we can repeat the construction, and we obtain in this way the following.



\begin{prop}\label{13-01-20e} There is a finitely supported random walk $(\mathbb D^n,\mu)$ on $\mathbb D^n$ such that the number of ergodic $K$-conformal measures on $\partial_M \mathbb D^n$ is at least $n+1$; and at least $n$ of these are non-atomic.
\end{prop}

By construction the random walk $(\mathbb D^n,\mu)$ has a spine in its Martin boundary. It follows that the KMS-spectrum of $\alpha^{\mu}$ is the whole line $\mathbb R$. For $\beta \notin \{0,1\}$ the simplex of $\beta$-KMS states is affinely homeomorphic to the tracial state of $C^*(\mathbb D^n)$. For $\beta = 1$ the tracial states of $C^*(\mathbb D^n)$ constitutes only a closed face of the simplex of $1$-KMS states; corresponding to those whose supporting $K$-conformal measure is concentrated on the spine.

\end{document}